\documentclass[11pt]{amsart}

\usepackage{amsmath}
\usepackage{amssymb}

\newcommand{\dst}{\displaystyle}
\newtheorem{thm}{Theorem}[section]

\newtheorem{lem}[thm]{Lemma}

\newtheorem{defi}[thm]{Definition}
\newtheorem*{A}{Conjecture}

\begin{document}

\title[{On row sequences of Pad\'e and Hermite-Pad\'e approximation}]
{On row sequences of Pad\'e and Hermite-Pad\'e approximation}

\author[G. L\'opez ]{G. L\'opez Lagomasino}
\address{Dpto. de Matem\'aticas\\
Escuela Polit\'ecnica Superior \\
Universidad Carlos III de Madrid \\
Universidad 30, 28911 Legan\'es, Spain} \email{lago@math.uc3m.es}

\dedicatory{Dedicated to my friend Ed on the occasion of his 70-th birthday}

\begin{abstract} A survey of direct and inverse type results for row sequences of Pad\'e and Hermite-Pad\'e approximation is given. A conjecture is posed on an inverse type result for type II Hermite-Pad\'e approximation when it is known that the sequence of common denominators of the approximating vector rational functions has a limit. Some inverse type results  are proved for the so called incomplete Pad\'e approximants which may lead to the proof of the conjecture and the connection is discussed.
\end{abstract}
\date{\today}
\maketitle


\vspace{0,5cm}\noindent


\section{Introduction}
The study of direct and inverse type results for sequences of rational functions with a fixed number of free poles has been a subject of constant interest in the research of E.B. Saff.  In different contexts (multi-point Pad\'e \cite{S3}, best rational \cite{S1}--\cite{S2}, Hermite-Pad\'e  \cite{GS1}--\cite{GS3}, and Pad\'e orthogonal approximations \cite{BLS}--\cite{BoLS}) such results are related with Montessus de Ballore's classical theorem \cite{Mon} on the convergence of the $m$-th row of the Pad\'e table associated with a formal Taylor expansion
\begin{equation} \label{i1} f(z) = \sum_{n\geq 0} \phi_n z^n
\end{equation}
provided that it represents a meromorphic function with exactly $m$ poles (counting multiplicities) in an open disk centered at the origin, and its converse due to A.A. Gonchar \cite[Section 3, Subsection 4]{gon2},  \cite[Section 2]{gon3} which allows to deduce analytic properties of $f$ if it is known that the poles of the approximants converge with geometric rate.

Let $m \in \mathbb{Z}_+ = \{0,1,2,\ldots\}$ be fixed. If $f$ is analytic at the origin, $D_m(f)$ denotes the largest open disk centered at the origin to which $f$ may be extended as a meromorphic function with at most $m$ poles and $R_m(f)$ is its radius; otherwise, we take $D_m(f) = \emptyset$ and $R_m(f) = 0$ for each $m \in \mathbb{Z}_+ $. The value $R_m(f)$ may be calculated,  as shown by J. Hadamard \cite{Had}, in terms of the Taylor coefficients $\phi_n$. Let $\mathcal{P}_m(f)$ be the set of  poles in $D_m(f)$. By $(R_{n,m})_{n \geq 0}, m \in \mathbb{Z}_+$ fixed, we denote the $m$-th row of the Pad\'e table associated with $f$,  see Definition \ref{defsimultaneos} restricted to $d=1$.

The combined Montessus de Ballore--Gonchar theorem may be formulated in the following terms
\begin{thm} \label{A}
 Let $f$ be a formal Taylor expansion about the origin and fix $m \in
\mathbb{N} = \{1,2,\ldots\}$. Then, the following two assertions are equivalent.
\begin{itemize}
\item[a)] $R_0(f) > 0$ and $f$ has exactly $m$ poles in $D_m(f)$
counting multiplicities.
\item[b)] There is a monic polynomial $Q_m$ of degree $m,\, Q_m(0) \neq 0,$  such
that the sequence of denominators $(Q_{n,m})_{ n \geq 0}$ of the
Pad\'e approximations of $f$, taken with leading coefficient equal to $1$, satisfies
$$
\limsup_{n \to \infty} \|Q_m - Q_{n,m}\|^{1/n} = \theta < 1,
$$
where $\|\cdot\|$ denotes the $\ell^1$ coefficient norm in the space of
polynomials.
\end{itemize}
Moreover, if either ${\rm a})$ or ${\rm b})$ takes place, the zeros of
$Q_m$ are the poles of $f$ in $D_m(f)$,
\begin{equation} \label{eq:6a}
\theta=\frac{\max \{|\xi|: \xi \in {\mathcal{P}}_m(f)\}}{R_m(f)},
\end{equation}
and
\begin{equation}\label{eq:6}
\limsup_{n \to \infty}
\|f - R_{n,m}\|_{K}^{1/n} = \frac{\|z\|_ {K}}{R_{m}(f)},
\end{equation}
where ${K}$ is any compact subset of $D_m(f) \setminus
{\mathcal{P}}_m(f)$.
\end{thm}

Since all norms in finite dimensional spaces are equivalent in  b)  any other norm in the $m+1$ dimensional space of polynomials of degree $\leq m$ would do as well.

From Theorem \ref{A} it follows that if $\xi$ is a pole of $f$ in
$D_m(f)$ of order $\tau$, then for each $\varepsilon
>0$, there exists $n_0$ such that for $n \geq n_0$, $Q_{n,m}$ has
exactly $\tau$ zeros in $\{z: |z - \xi |< \varepsilon\}$. We say
that each pole of $f$ in $D_m(f)$ attracts as many zeros of
$Q_{n,m}$ as its order when $n$ tends to infinity.

Under assumptions a), in \cite{Mon} Montessus de Ballore proved that
\[ \lim_{n \to \infty}  Q_{n,m} = Q_m,
\qquad \lim_{n\to \infty} R_{n,m} = f,
\]
with uniform convergence on compact subsets of $D_m(f) \setminus
{\mathcal{P}}_m(f)$ in the second limit. In essence, Montessus
proved that a) implies b), showed that
$\theta \leq {\max \{|\xi|: \xi \in {\mathcal{P}}_m(f)\}}/{R_m(f)}$,
and proved \eqref{eq:6} with equality  replaced by $\leq$. These are
the so called direct statements of the theorem. The inverse
statements, b) implies a), $\theta \geq {\max \{|\xi|: \xi \in
{\mathcal{P}}_m(f)\}}/{R_m(f)}$, and the inequality $\geq$ in
\eqref{eq:6} are immediate consequences of \cite[Theorem 1]{gon2}.
The study of inverse problems when the behavior of individual sequences of poles of the approximants is known was suggested
by A.A. Gonchar in \cite[Subsection 12]{gon2} where he presented
some interesting conjectures. Some of them were solved in
\cite{sue1} and \cite{sue2} by S.P. Suetin.

In \cite{GS1}, Graves-Morris and Saff proved an analogue of
Montessus' theorem for Hermite-Pad\'e (vector rational) approximation with
the aid of the concept of polewise independence of a system of
functions.

Let ${\bf f} = (f_1,\ldots,f_d)$ be a system
of $d$ formal or convergent Taylor expansions about the origin; that
is, for each $k=1,\dots,d$, we have
\begin{equation} \label{sistema} f_k(z) = \sum_{n=0}^{\infty} \phi_{n,k} z^n,
\qquad \phi_{n,k} \in {\mathbb{C}}.
\end{equation}
Let  $\mathbf{D}=\left(D_1,\dots,D_d\right)$ be a system of domains
such that, for each $k=1,\dots,d,$ $f_k$ is meromorphic in $D_k$. We
say that the point $\xi$  is a pole of ${\bf f}$ in $\mathbf{D}$ of
order $\tau$ if there exists an index $k\in\{1,\dots,d\}$  such that
$\xi\in D_k$ and it is a pole of $f_k$ of order $\tau$, and for
$j\not = k$ either $\xi$ is a pole of $f_j$ of order less than or
equal to $\tau$ or $\xi \not\in D_j$. When
$\mathbf{D}={(D,\dots,D)}$ we say that  $\xi$ is a pole of ${\bf f}$
in $D$.

Let $R_0({\bf f})$ be radius of
the largest open disk $D_0({\bf f})$ in which all the expansions $f_k, k=1,\ldots,d$
correspond to analytic functions. If $R_0({\bf f}) =0$, we take
$D_{m}({\bf f}) = \emptyset, m \in {\mathbb{Z}}_+$; otherwise,
$R_m({\bf f})$ is the radius of the largest open disk $D_{m}({\bf f})$
centered at the origin to which all the analytic elements $(f_k,
D_0(f_k))$ can be extended so that ${\bf f}$ has at most $m$ poles
counting multiplicities. The disk $D_{m}({\bf f})$ constitutes for
systems of functions the analogue of the $m$-th disk of meromorphy
defined by J. Hadamard in \cite{Had} for $d=1$. Moreover, in that
case both definitions coincide.

By $\mathcal{Q}_m(\mathbf{f})$ we denote the monic polynomial whose
zeros are the poles of ${\bf f}$ in $D_{m}({\bf f})$ counting
multiplicities. The set of distinct zeros of
$\mathcal{Q}_m(\mathbf{f})$ is denoted by ${\mathcal{P}}_m({\bf
f})$.

\begin{defi}\label{defsimultaneos}
Let ${\bf f} = (f_1,\ldots,f_d)$ be a system of $d$ formal Taylor
expansions as in \eqref{sistema}. Fix a multi-index ${\bf m} =
({m_1},\ldots,m_d) \in {\mathbb{Z}}_+^d \setminus \{\bf 0\}$ where
${\bf 0}$ denotes the zero vector in ${\mathbb{Z}}_+^d$. Set $|{\bf
m}| = {m_1} +\cdots + m_d$. Then, for each $n \geq \max
\{{m_1},\ldots,m_d\}$, there exist polynomials $Q, P_k,
k=1,\ldots,d,$ such that
\begin{itemize}
\item[a.1)] $\deg P_k \leq n - m_k,\, k=1,\ldots,d,\quad \deg Q
\leq |\mathbf{m}|,\quad Q \not\equiv 0,$
\item[a.2)] $Q(z) f_k(z) - P_k (z) = A_k z^{n+1} + \cdots .$
\end{itemize}
The vector rational function ${\bf R}_{n,{\bf m}} = (P_1 /Q
,\ldots,P_d /Q)$ is called an $(n,{\bf m})$ (type II) Hermite-Pad\'e
approximation of ${\bf f}$.
\end{defi}

Type I  and type II Hermite-Pad\'e approximation were introduced by Ch. Hermite  and used in the proof of the transcendence of $e$, see \cite{Herm}. We will only consider here type II and, for brevity, will be called Hermite-Pad\'e approximants.

In contrast with Pad\'e approximation, such vector rational approximants, in general, are not
uniquely determined  and in the sequel we assume that
given $(n,{\bf m})$ one particular solution is taken. For that
solution we write
\begin{equation} \label{incomplete}
{\bf R}_{n,{\bf m}} = (R_{n,{\bf m},1},\ldots,R_{n,{\bf m},d})=
(P_{n,{\bf m},1},\ldots,P_{n,{\bf m},d})/Q_{n,{\bf m}},
\end{equation}
where $Q_{n,{\bf m}}$ is a monic polynomial that has no common zero
simultaneously with all the $P_{n,{\bf m},k}$. Sequences $({\bf
R}_{n,\bf m})_{ n \geq |{\bf m}|},$ for which ${\bf m}$ remains fixed when $n$ varies
are called row sequences.

For each $r>0$, set $D_r=\{z\in\mathbb{C}\,:\,|z|<r\}$,
$\Gamma_r=\{z\in\mathbb{C}\,:\,|z|=r\}$, and
$\overline{D}_r=\{z\in\mathbb{C}\,:\,|z|\le r\}$.
\begin{defi} \label{polewisei}
Let ${\bf f} = (f_1,\ldots,f_d)$ be a system of meromorphic
functions in the disk $D_r$ and let ${\bf m} = (m_1,\ldots,m_d)\in
\mathbb{Z}_+^d \setminus \{{\bf 0}\}$. We say that the system ${\bf
f}$ is polewise independent with respect to $\bf m$ in $D_r$ if
there do not exist polynomials $p_1,\ldots,p_d$, at least one of
which is non-null, such that
\begin{itemize}
\item[b.1)] $\deg p_k < m_k$ if $m_k\ge 1,\, k=1,\ldots,d$,
\item[b.2)] $ p_k \equiv 0$ if $m_k=0,\, k=1,\ldots,d$,
\item[b.3)] $\sum_{k=1}^d p_k f_k$ is analytic on  ${D}_{r}$.
\end{itemize}
\end{defi}

In \cite[Theorem 1]{GS1}, Graves-Morris and Saff established an analogue of the direct part of the previous theorem when ${\bf f}$ is polewise independent with respect to $\bf m$ in $D_{|\bf m|}({\bf f})$ obtaining upper bounds
for the convergence rates corresponding to \eqref{eq:6a} and
\eqref{eq:6}. It should be stressed that \cite{GS1} was pioneering in the sense that it initiated a convergence theory for Hermite-Pad\'e approximation.

The result \cite[Theorem 1]{GS1} does not allow a converse statement in the sense of Gonchar's theorem.
Inspired in the concept of polewise independence, in \cite{CCL} we proposed  the following relaxed version of it.

\begin{defi} \label{e-pole}
Given ${\bf f} = (f_1,\ldots,f_d)$ and ${\bf m} =
(m_1,\ldots,m_d)\in \mathbb{Z}_+^d \setminus \{{\bf 0}\}$ we say
that $\xi\in\mathbb{C}\setminus\{0\}$ is a system pole of order
$\tau$ of ${\bf f}$ with respect to $\bf m$ if $\tau$ is the largest positive integer such that for each
$s=1,\dots,\tau$ there exists at least one polynomial combination of
the form
\begin{equation}\label{combination}
\sum_{k=1}^d p_k f_k,\quad \deg p_k<m_k,\quad k=1,\dots,d,
\end{equation}
which is analytic on a neighborhood of $\overline{D}_{|\xi|}$ except
for a pole at $z=\xi$ of exact order $s$.  If some component $m_k$
equals zero the corresponding polynomial $p_k$ is taken identically
equal to zero.
\end{defi}

The advantage of this definition with respect to that of
polewise independence is that it does not require to determine
a priori a region where the property should be verified. Polewise independence of $\bf f$ in $D_{|{\bf m}|}({\bf f})$ with respect to $\bf m$ implies that $\bf f$ has in $D_{|{\bf m}|}$ exactly $|{\bf m}|$ system poles (counting their order).

We wish to underline that if some component $m_k$ equals zero, that
component places no restriction on Definition \ref{defsimultaneos}
and does not report any benefit in finding system poles; therefore,
without loss of  generality one can restrict the attention to
multi-indices ${\bf m} \in \mathbb{N}^d$.

A system $\mathbf{f}$ cannot have
more than $|\mathbf{m}|$ system poles with respect to $\mathbf{m}$
counting their order. A system pole need not be a pole of ${\bf f}$
and a pole may not be a system pole, see examples in \cite{CCL}.

To each system pole $\xi$ of $\mathbf{f}$ with respect to
$\mathbf{m}$ one can associate several characteristic values. Let $\tau$ be
the order
of $\xi$ as a system pole of $\mathbf{f}$. For each $s=1,\dots,\tau$
denote by $r_{\xi,s}(\mathbf{f},\mathbf{m})$ the largest of all the
numbers $R_s(g)$ (the radius of the largest disk containing at most
$s$ poles of $g$), where $g$ is a polynomial combination of type
\eqref{combination} that is analytic on a neighborhood of
$\overline{D}_{|\xi|}$ except for a pole at $z=\xi$ of order $s$.
Set
$$
\begin{array}{c}
\dst R_{\xi,s}(\mathbf{f},\mathbf{m}):=\min_{k=1,\dots,s} r_{\xi,
{k}}(\mathbf{f},\mathbf{m}),
\\ \\ \dst
R_\xi(\mathbf{f},\mathbf{m}) :=
R_{\xi,\tau}(\mathbf{f},\mathbf{m})=\min_{s=1,\dots,\tau}
r_{\xi,s}(\mathbf{f},\mathbf{m}).
\end{array}
$$

Obviously, if $d=1$ and $({\bf f,\bf m}) = (f,m)$, system poles and
poles in $D_m(f)$ coincide. Also,
$R_\xi(\mathbf{f},\mathbf{m})=R_{m}(f)$ for each
 pole $\xi$ of $f$ in $D_{m}(f)$.

Let $\mathcal{Q}(\mathbf{f},\mathbf{m})$ denote the
monic polynomial whose zeros are the system poles of ${\bf f}$ with
respect to $\mathbf{m}$ taking account of their order. The set of
distinct zeros of $\mathcal{Q}(\mathbf{f},\mathbf{m})$ is denoted by
${\mathcal{P}}({\mathbf f},\mathbf{m})$. We have (see \cite[Theorem 1.4]{CCL})

\begin{thm} \label{reciprocal} Let ${\bf f}$ be a system
of formal Taylor expansions as in \eqref{sistema} and fix a
multi-index $\mathbf{m}\in\mathbb{N}^d $. Then, the following
assertions are equivalent.
\begin{itemize}
\item[a)]
$R_0(\mathbf{f}) > 0$ and ${\bf f}$ has exactly $|{\bf m}|$ system
poles with respect to $\mathbf{m}$ counting multiplicities.
\item[b)] The
denominators $Q_{n,{\bf m}},\,n \geq |{\bf m}|,$ of simultaneous
Pad\'e approximations of ${\bf f}$ are uniquely determined for all
sufficiently large $n$ and there exists a polynomial
$Q_{|\mathbf{m}|}$ of degree $|\mathbf{m}|,\, Q_{|\mathbf{m}|}(0)
\neq 0,$ such that
$$
\limsup_{n \to \infty} \|Q_{|{\bf m}|} - Q_{n,{\bf m}}\|^{1/n} =
\theta < 1.
$$
\end{itemize}
Moreover, if either ${\rm a})$ or ${\rm b})$ takes place then
$Q_{|\mathbf{m}|}\equiv\mathcal{Q}(\mathbf{f},\mathbf{m})$ and
\begin{equation} \label{eq:6b}
\theta=\max\left\{\frac {|\xi|}
{R_{\xi}(\mathbf{f},\mathbf{m})}\,:\, \xi \in
{\mathcal{P}}(\mathbf{f},\mathbf{m}) \right\}.
\end{equation}
\end{thm}

If $d=1$,  $R_{n,m}$ and $Q_{n,m}$ are uniquely determined; therefore,  Theorem \ref{reciprocal} contains Theorem \ref{A}. The analogue of  \eqref{eq:6}
is found in \cite[Theorem 3.7]{CCL}).

In the rest of the paper we wish to discuss the case when
\begin{equation}
\label{condfun}
\lim_{n \to \infty} Q_{n,{\bf m}} = Q_{|{\bf m}|}, \qquad \deg Q_{|{\bf m}|} = |{\bf m}|, \qquad Q_{|{\bf m}}(0) \neq 0,
\end{equation}
but the rate of convergence is not known in advance. Now the reference in the scalar case is a result by S.P. Suetin \cite{sue2}.

\begin{thm} \label{Sueteo}
Assume that $\lim_{n\to \infty} Q_{n,m}(z) = Q_m(z) = \prod_{k=1}^m(z -z_k)$ and
\[ 0 < |z_1| \leq \cdots \leq |z_N| < |z_{N+1}| = \cdots = |z_m| = R.
\]
Then $z_1,\ldots,z_N$ are the poles of $f$ in $D_{m-1}(f)$ (taking account of their order), $R_N(f) = \cdots = R_{m-1}(f) = R$, and $z_{N+1}, \ldots, z_m$ are singularities of $f$ on the boundary of $D_{m-1}(f)$.
\end{thm}

When $m=1$ it is easy to see from the definition that $Q_{n,1} = z -  (\phi_n/\phi_{n+1})$ whenever $\phi_{n+1} \neq 0$. Therefore, Suetin's theorem contains the classical theorem of E. Fabry \cite{Fabry} which states that $\lim_{n \to \infty} \phi_n/\phi_{n+1} = \zeta \neq 0$ implies that $R_0(f) = |\zeta|$ and $\zeta$ is a singular point of $f$.

Let us introduce the concept of system singularity of  $\bf f$ with respect to $\bf m$.

\begin{defi} \label{e-pole}
Given ${\bf f} = (f_1,\ldots,f_d)$ and ${\bf m} =
(m_1,\ldots,m_d)\in \mathbb{Z}_+^d \setminus \{{\bf 0}\}$ we say
that $\xi\in\mathbb{C}\setminus\{0\}$ is a system singularity  of ${\bf f}$ with respect to $\bf m$ if there exists at least one polynomial combination $F$ of
the form \eqref{combination} analytic on  $ {D}_{|\xi|}$ and $\xi$ is a singular point of $F$.
\end{defi}

We believe that the following result holds.

\begin{A}
\label{conjecture}
Assume that $Q_{n,\bf m}$ is unique for all sufficiently large $n$, \eqref{condfun} takes place, and let $Q_{|{\bf m}|}(\zeta) =0$. Then, $\zeta$ is a system singularity of $\bf f$ with respect to $\bf m$. If $\zeta \in D_1(F)$, for some polynomial combination $F$ determines the system singularity of $\bf f$ at $\zeta$, then $\zeta$ is a system pole of $\bf f$ with respect to $\bf m$ of order equal to the multiplicity of $\zeta$ as a zero of $Q_{|{\bf m}|}$.
\end{A}

This conjecture applied to the scalar case reduces to Theorem \ref{Sueteo}.

In Section 2, we give a result similar to Theorem \ref{Sueteo} for so called incomplete Pad\'e approximation. Such approximants were introduced in \cite{CCL0} and used in \cite{CCL} to prove Theorem \ref{reciprocal}. In the final section we describe some steps which may lead to the proof of the conjecture.


\section{Incomplete Pad\'e approximants}\label{incompletos}

Consider the following construction.

\begin{defi}\label{defincompletos}
Let $f$  denote a formal Taylor expansion as in \eqref{i1}. Fix
$m\ge m^*\ge 1$. Let $n \geq m$. We say that the rational function
$r_{n,m}$ is an incomplete Pad\'e approximation of type $(n,m,{m^*})
$ corresponding to $f$ if $r_{n,m}$ is the quotient of any two
polynomials $p $ and $q $ that verify
\begin{itemize}
\item[c.1)]
$\deg p \le n-{m^*},\quad \deg q \le m,\quad q \not\equiv 0,$
\item[c.2)] $q(z) f(z)-p(z)=
A z^{n+1}+ \cdots .$
\end{itemize}
\end{defi}

Given $(n,m,{m^*}), n \geq m \geq {m^*},$  the
Pad\'e approximants $R_{n,{m^*}},\ldots,R_{n,m}$ can all be regarded
as incomplete Pad\'e approximation of type $(n,m,{m^*})$ of $f$.
From Definition \ref{defsimultaneos} and
(\ref{incomplete}) it follows that $R_{n,{\bf m},k}, k=1,\ldots,d,$
is an incomplete Pad\'e approximation of type $(n,|\mathbf{m}|,m_k)$
with respect to $f_k$.

In the sequel, for each $n \geq m \geq {m^*},$ we choose one incomplete Pad\'e approximant. After canceling out common factors between $q$ and $p$,
we write $ r_{n,m} = p_{n,m}/q_{n,m}, $ where, additionally,
$q_{n,m}$ is normalized as follows
\begin{equation} \label{normado} q_{n,m}(z) = \prod_{|\zeta_{n,k}|
\leq 1} \left(z - \zeta_{n,k}\right) \prod_{|\zeta_{n,k}| > 1}
\left(1- \frac{z}{\zeta_{n,k}}\right).
\end{equation}
Suppose that $q$ and $p$ have a
common zero at $z=0$ of order $\lambda_n$. Notice that $0 \leq \lambda_n \leq m$. From c.1)-c.2) it follows
that
\begin{itemize}
\item[c.3)] $ \deg p_{n,m} \leq n-m^*-\lambda_n, \quad
\deg q_{n,m} \leq m-\lambda_n, \quad q_{n,m} \not\equiv 0,$
\item[c.4)] $q_{n,m}(z) f(z)-p_{n,m}(z)=
A z^{n+1-\lambda_n}+ \cdots .$
\end{itemize}
where $A$ is, in general, a different constant from the one in c.2).

From Definition \ref{defincompletos} it readily follows that for each $n \ge m$
\begin{equation}
\label{tele}
 r_{n+1,m}(z) - r_{n,m}(z) = \frac{A_{n,m} z^{n+1-\lambda_n-
\lambda_{n+1}}q^*_{n,m-m^*}(z)}{q_{n,m}(z)q_{n+1,m}(z)},
\end{equation}
where $A_{n,m}$ is some constant and $q^*_{n,m-m^*}$ is a polynomial
of degree less than or equal to $m - m^*$ normalized as in \eqref{normado}.

The first difficulty encountered in dealing with inverse-type
results is to justify in terms of the data that the formal series
corresponds to an analytic element around the origin which does not reduce to a
polynomial. Set
\[ R_m^*(f) = \left(\limsup_{n \to \infty} |A_{n,m}|^{1/n}\right)^{-1}, \qquad D_m^*(f) = \{z: |z| < R_m^*(f)\}.
\]

Let $B$ be a subset of the complex plane $\mathbb{C}$. By
$\mathcal{U}(B)$ we denote the class of all coverings of $B$ by at
most a numerable set of disks. Set
$$
\sigma(B)=\inf\left\{\sum_{i=1}^\infty
|U_i|\,:\,\{U_i\}\in\mathcal{U}(B)\right\},
$$
where $|U_i|$ stands for the radius of the disk $U_i$. The quantity
$\sigma(B)$ is called the $1$-dimensional Hausdorff content of the
set $B$. In the papers we refer to below, the only properties used of the $1$-dimensional Hausdorff content follow easily from the definition. They are: subadditivity, monotonicity, and that the  $1$-dimensional Hausdorff content of a disk of radius $R$ and a segment of length $d$ are $R$ and $d/2$, respectively.

Let $(\varphi_n)_{n\in\mathbb{N}}$ be a sequence of functions
defined on a domain $D\subset\mathbb{C}$ and $\varphi$ another
function defined on $D$. We say that
$(\varphi_n)_{n\in\mathbb{N}}$ converges in $\sigma$-content to
the function $\varphi$ on compact subsets of $D$ if for each compact
subset $K$ of $D$ and for each $\varepsilon
>0$, we have
$$
\lim_{n\to\infty} \sigma\{z\in K :
|\varphi_n(z)-\varphi(z)|>\varepsilon\}=0.
$$
We denote this writing $\sigma$-$\lim_{n\to \infty} \varphi_n =
\varphi$ inside $D$.

Using telescopic sums, it is not difficult to prove the following (see \cite[Theorem 3.4]{CCL0}).

\begin{lem} \label{teo:5} Let $f$ be a formal power series as in \eqref{i1}. Fix $m$
and ${m^*}$ nonnegative integers, $m \geq {m^*}$. Let
$(r_{n,m})_{n \geq m}$ be a sequence of incomplete Pad\'e
approximants of type $(n,m,{m^*})$ for $f$. If ${R}_{m}^*(f) > 0$
then $R_0(f) > 0$. Moreover,
\[D_{m^*}(f) \subset D_m^*(f) \subset D_m(f)\]
and $D_{m}^*(f)$ is the largest disk in compact subsets of which $\sigma-lim_{n\to \infty} r_{n,m} = f$. Moreover,  the
sequence $(r_{n,m})_{n\ge m}$ is pointwise divergent in $\{z:|z| > R_m^*(f)\}$ except on a set of $\sigma$-content zero.
\end{lem}

We also have (see \cite[Corollaries 2.4, 2.5]{CCL})

\begin{lem} \label{cor;1} Let $f$ be a formal power series
as in \eqref{i1}. Fix $m \geq m^* \geq 1$. Assume that there exists
a polynomial ${q}_m$ of degree greater than or equal to
 $ m-m^*+1,\,{q}_m(0) \neq 0,$ such that $\lim_{n
\to \infty} {q}_{n,m} = {q}_m$. Then $0<R_0(f)<\infty$ and the zeros
of ${q}_{m}$ contain all the poles, counting multiplicities, that
$f$ has in $D^*_{m}(f)$.
\end{lem}

Suppose that $\limsup_n |A_{n,m}|^{1/n} = 1$. It is known,  that there exists a regularizing sequence $(A_{n,m}^*)_{n \geq m}$ such that:
\begin{itemize}
\item[i)] $\lim_{n\to \infty}A_{n,m}^*/A_{n+1,m}^* = 1$,
\item[ii)] $\{\log(A_{n,m}^*/n!)\}$ is concave,
\item[iii)] $|A_{n,m}| \leq |A_{n,m}^*|, n \in \mathbb{Z}_+$,
\item[iv)] $|A_{n,m}| \geq c|A_{n,m}^*|, n \in \Lambda \subset \mathbb{Z}_+, c > 0$ for an infinite sequence $\Lambda$.
\end{itemize}

The use of such regularizing sequences is well established in the study of singularities of Taylor series (see, for example, \cite{Agm} and \cite{Man}). Its use was extended by S.P. Suetin in \cite{sue2} to Pad\'e approximation for proving Theorem \ref{Sueteo}. The proofs of  \cite[Lemmas 1, 2]{sue2} (see also \cite[Chapter 1]{sue3})  may be easily adjusted to produce the following result for incomplete Pad\'e approximation.

\begin{lem} \label{lem:inc} Let $f$ be a formal power series
as in \eqref{i1}. Fix $m \geq m^* \geq 1$. Assume that $\lim_{n \to \infty} |A_{n,m}|^{1/n} = 1$. For any $\delta > 0$
\begin{equation} \label{a}
\max_{|z| \geq e^{\delta}} |p_{n,m}(z)/(A_{n,m}^* z^n)| = \mathcal{O}(1), \qquad n \to \infty.
\end{equation}
For every compact $K \subset \{z:|z| < e^{-\delta}\} \setminus \mathcal{P}(f)$,
\begin{equation} \label{4} \max_K |(q_{n,m}f - p_{n,m})(z)/(A_{n,m}^* z^n)| = \mathcal{O}(1), \qquad n \to \infty.
\end{equation}
Assume that there exists a polynomial $q_m, \deg q_m = m, q_m(0) \neq 0,$ such that
\[ \lim_{n \to \infty} q_{n,m} = q_m.
\]
Let $f$ be holomorphic in some region $G \supset D_m^*(f) \setminus \mathcal{P}(f)$. Then, for every compact $K \subset G$, \eqref{4} takes place.
\end{lem}

In the sequel $\mbox{dist}(\zeta,B_n)$ denotes the distance from a point $\zeta$ to a set $B_n$. Let $\mathcal{P}_{n,m}(f) = \{\zeta_{n,1},\ldots,\zeta_{n,m_n}\}$ be the set of zeros of $q_{n,m}$ and the points are enumerated so that
\[ |\zeta_{n,1}-\zeta| \leq \cdots \leq |\zeta_{n,m_n}- \zeta|.\]
We say that $\lambda = \lambda(\zeta)$ points of $\mathcal{P}_{n,m}$ tend to $\zeta$ if
\[ \lim_{n \to \infty} |\zeta_{n,\lambda}-\zeta| = 0, \qquad \limsup_{n \to \infty} |\zeta_{n,\lambda+1}-\zeta| > 0.
\]
By convention, $\limsup_{n \to \infty} |\zeta_{n,\kappa}-\zeta| > 0$ for $\kappa > \liminf_{n \to \infty} m_n .$

\begin{thm} \label{teo:1} Let $f$ be a formal power series
as in \eqref{i1}. Fix $m \geq m^* \geq 1$. Assume that $0 < R_m^*(f) < +\infty$. Suppose that
\[ \lim_{n \to \infty} \mbox{\rm dist} (\zeta,\mathcal{P}_{n,m}(f)) = 0.
\]
Let $\mathcal{Z}_n(f)$ be the set of zeros of $q^*_{n,m-m^*}$. If $|\zeta| > R_m^*(f)$, then
\begin{equation} \label{3} \lim_{n \in \Lambda} \mbox{\rm dist} (\zeta,\mathcal{Z}_n(f)) = 0
\end{equation}
where $\Lambda$ is any infinite sequence of indices verifying ${\rm iv)}$ in the  regularization of $(A_{n,m})_{n\geq m}$.
If $|\zeta| < R_m^*(f)$ , then either \eqref{3} takes place or $\zeta$ is a pole of $f$ of order greater or equal to $\lambda(\zeta)$.
If $\lim_{n \to \infty} q_{n,m} = q_m, \deg q_m = m, q_m(0) \neq 0,$ and $|\zeta| = R_m^*(f)$ then we have either \eqref{3} or $\zeta$ is a singular point.  If the zeros of $q_m$ are distinct then at least $m^*$ of them are singular points of $f$ and lie in the closure of $D_m^*(f)$, those lying in $D_m^*(f)$ are simple poles.
\end{thm}

\begin{proof} Without loss of generality, we can assume that $R_m^*(f) = 1$. The general case reduces to it with the change of variables $z \to z/R_m^*(f)$.
Assume that  $|\zeta| \neq 1$ and $\zeta$ is a regular point of $f$ should  $|\zeta| < 1$. Choose $\delta > 0$ such that $|\zeta| > e^{\delta}$ or $|\zeta| < e^{-\delta}$ depending on whether $|\zeta| > 1$ or $|\zeta| < 1$, respectively. Let $q_{n,m}(\zeta_n) = 0, \lim_{n \to \infty} \zeta_n = \zeta$.

Evaluating at $\zeta_n$, using \eqref{a}, if  $|\zeta| > 1$ or \eqref{4}, when  $|\zeta| < 1$, and taking ${\rm iv)}$ into account, it follows  that
\[ |p_{n,m}(\zeta_n)/(A_{n,m} \zeta_n^n)| \leq C_1,  \qquad n \geq n_0, \qquad n \in \Lambda,
\]
where $C_1$ is some constant and $\Lambda$ is the sequence of indices which appears in the regularization of $(A_{n,m})_{n \geq m}$. (In the sequel $C_1, C_2,\ldots$ denote constants which do not depend on $n$.) However, from \eqref{tele} it follows that
\[ p_{n,m}(\zeta_n)/(A_{n,m} \zeta_n^n) = -\zeta_n^{1 -\lambda_n - \lambda_{n+1}} q^*_{n,m-m^*}(\zeta_n)/q_{n+1,m}(\zeta_n),
\]
which combined with the previous inequality gives
\[ |q^*_{n,m-m^*}(\zeta_n)| \leq C_2 |q_{n+1,m}(\zeta_n)|, \qquad n \geq n_0, \qquad n \in \Lambda.
\]
Therefore, \eqref{3} takes place.

If $|\zeta| = 1$ and $\zeta$ is a regular point the proof of \eqref{3} is the same as for the case when $|\zeta| < 1$. In this case use \eqref{4} on a closed neighborhood of $\zeta$ contained in $G \supset D_m^*(f) \setminus \mathcal{P}(f)$.

Now, assume that $|\zeta| < 1$ and $\limsup_{n \in \Lambda} \mbox{\rm dist} (\zeta,\mathcal{Z}_n(f)) > 0$. Then, $\zeta$ is a singular point of $f$. Since  $ D_m^*(f)\subset D_m(f)$ according to Lemma \ref{teo:5},  $\zeta$ must be a pole of $f$. Let $\tau$ be the order of the pole of $f$ at $\zeta$.  Set $w(z) = (z-\zeta)^{\tau}$ and $F = w f$. Notice that $F(\zeta) \neq 0$. Using \eqref{4} and $\rm iv)$, it follows that there exists a closed disk $U_r$ centered at $\zeta$ of radius $r$ sufficiently small so that
\begin{equation} \label{4^*} \max_{U_r} |(q_{n,m}F - p_{n,m} w)(z)/(A_{n,m} z^n)| \leq C_3.\qquad n \geq n_0, \qquad n \in \Lambda.
\end{equation}

Suppose that $\tau  < \lambda(\zeta)$. Since $\sigma-lim_{n\to \infty} r_{n,m} = f$ (see Lemma \ref{teo:5}), it follows that for each $n \in \mathbb{Z}_+$ there exists a zero of $\eta_n$ of $p_{n,m}$ such that $\lim_{n \to \infty} \eta_n = \zeta$. Take $r > 0$ sufficiently small so that $\min_{U_r} |F(z)| >0$. Substituting $\eta_n$ in \eqref{4^*}, we have
\[|q_{n,m}(\eta_n)/(A_{n,m} \eta_n^n)| \leq C_4, \qquad n \geq n_0, \qquad n \in \Lambda, \]
and taking into account that \eqref{tele} leads to
\[ q_{n,m}(\eta_n)/(A_{n,m} \eta_n^n) = \eta_n^{1- \lambda_n - \lambda_{n+1}} q_{n,m-m^*}^*(\eta_n)/ p_{n+1,m}(\eta_n),
\]
we obtain
\[  |q_{n,m-m^*}^*(\eta_n)| \leq C_5 | p_{n+1,m}(\eta_n)| , \qquad n \geq n_0, \qquad n \in \Lambda.
\]
Since $\limsup_{n \in \Lambda} \mbox{\rm dist} (\zeta,\mathcal{Z}_n(f)) > 0$, it follows that
\begin{equation}
\label{Fsub} \lim_{n \in \Lambda' } |p_{n+1,m}(\eta_n)| > 0.
\end{equation}
for some subsequence $\Lambda' \subset \Lambda$.

The normalization \eqref{normado} imposed on $q_{n,m} $ implies that for any compact $K \subset \mathbb{C}$ we have $\sup_n \max_K |q_{n,m}(z)| \leq C_6 $. So, any sequence $(q_{n,m})_{n \in I}, I \subset \mathbb{Z}_+,$ contains a subsequence which converges uniformly on any compact subset of $\mathbb{C}$. This, combined with $\sigma-lim_{n\to \infty} r_{n,m} = f$ in $D_m^*(f)$, and the assumption that $\tau < \lambda(\zeta)$ imply that there exists a sequence of indices $\Lambda'' \subset \Lambda'$ such that $\lim_{n \in \Lambda''} p_{n+1,m} = F_1$ uniformly on a closed neighborhood of $\zeta$, where $F_1$ is analytic at $\zeta$ and $F_1(\zeta) = 0$ (see  \cite[Lemma 1]{gon1} where it is shown that under adequate assumptions uniform convergence on compact subsets of a region can be derived from convergence in $1$-dimensional Hausdorff content). This contradicts \eqref{Fsub}. Thus, $\tau \geq \lambda(\zeta)$ as claimed.

To complete the proof recall that $\deg q^*_{n,m-m^*} \leq m- m^*$ for all $n \geq m$. In particular, this implies that for each $n \in \Lambda$ the set $\mathcal{Z}_n(f)$ has at most $m-m^*$ points. Each zeros $\zeta$ of $q_m$ such that either $|\zeta| > 1$ or $|\zeta|\leq 1$ and is regular attracts a sequence of points in $\mathcal{Z}_n(f), n \in \Lambda$. This is clearly impossible if the total number $M$ of such zeros of $q_m$ exceeds $m -m^*$. So, $M \leq  m-m^*$. The complement is made up of zeros of $q_m$ which are singular and lie in the closure of $D_m^*(f)$. Those lying in $D_m^*(f)$ are simple poles according to Lemma \ref{cor;1}.
\end{proof}


\section{Simultaneous approximation}\label{simultaneos}
Throughout this section, $\mathbf{f}=(f_1,\dots,f_d)$ denotes a
system of formal power expansions as in \eqref{sistema} and
$\mathbf{m}=(m_1,\dots,m_d)\in\mathbb{N}^d$ is a fixed multi-index.
We are concerned with the simultaneous approximation of $\textbf{f}$
by sequences of vector rational functions defined according to
Definition \ref{defsimultaneos} taking account of
(\ref{incomplete}). That is, for each $n\in\mathbb{N}, n \geq |{\bf
m}|$, let $\left(R_{n,\mathbf{m},1},\dots,R_{n,\mathbf{m},d}\right)$
be a Hermite-Pad\'e approximation of type $(n,\mathbf{m})$
corresponding to $\mathbf{f}$.

As we mentioned earlier, $R_{n,\mathbf{m},k}$ is an incomplete
Pad\'e approximant of type $(n,|\mathbf{m}|,m_k)$ with respect to
$f_k,\,k=1,\dots,d$. Thus, from Lemma \ref{teo:5}
$$
D_{m_k}(f_k)\subset D^*_{|\mathbf{m}|}(f_k) \subset
D_{|\mathbf{m}|}(f_k),\quad k=1,\dots,d.
$$

\begin{defi} A vector ${\bf f} = (f_1,\ldots,f_d)$ of formal power
expansions is said to be polynomially independent with respect to
${\bf m} = (m_1,\ldots,m_d) \in {\mathbb{N}}^d$ if there do not
exist polynomials $p_1,\ldots,p_d$, at least one of which is
non-null, such that
\begin{itemize}
\item[d.1)] $\deg p_k < m_k ,\, k=1,\ldots,d$,
\item[d.2)] $\sum_{k=1}^d p_k f_k$ is a polynomial.
\end{itemize}
\end{defi}
In particular, polynomial independence implies that for each
$k=1,\ldots,d$, $f_k$ is not a rational function with at most $m_k-1$ poles.
Notice that polynomial independence
may be verified solely in terms of the coefficients of the formal
Taylor expansions defining the system $\mathbf{f}$. The system ${\bf f}$ is
polynomially independent with respect to ${\bf m}$ if for all
$n\geq n_0$ the polynomial $ Q_{n,{\bf m}}$ is unique and $\deg
Q_{n,{\bf m}} =|{\bf m}|$, see \cite[Lemma 3.2]{CCL}.

An approach to the proof of the conjecture could be
\begin{itemize}
\item Remove the restriction in the last part of Theorem \ref{teo:1} that the zeros of $q_m$ are distinct.
\item Assuming \eqref{condfun}, apply the improved version of  Theorem \ref{teo:1} to the components of $\bf f$.
\item Using polynomial combinations of the form \eqref{combination} prove that each zero of $Q_{|{\bf m}|}$ is a system singularity. It is sufficient to consider multi-indices of the form ${\bf m} = (1,1,\ldots,1)$ (see beginning of \cite[Section 3]{CCL} for the justification); then, \eqref{combination} reduces to linear combinations.
\item Prove the last part of the conjecture using the final statement of Lemma \ref{cor;1}.
\end{itemize}


\end{document}